\documentclass[12pt]{amsart}
\usepackage{amssymb}
\usepackage{mathrsfs}
\usepackage{amsxtra}
\usepackage{graphics}
\usepackage{latexsym}
\usepackage{xcolor}
\usepackage{diagrams}
\usepackage{amsmath}
\usepackage{amssymb,amsthm,amsfonts}
\usepackage{mathtools}
\usepackage{amscd}
\usepackage[arrow, matrix, curve]{xy}
\usepackage{syntonly}
\title[Twisted functoriality]{Twisted functoriality in nonabelian Hodge theory in positive characteristic}

\author[Mao Sheng]{Mao Sheng}
\email{msheng@ustc.edu.cn}
\address{School of Mathematical Sciences,
	University of Science and Technology of China, Hefei, 230026, China}

\begin{document}
\theoremstyle{plain}
\newtheorem{thm}{Theorem}[section]
\newtheorem{theorem}[thm]{Theorem}
\newtheorem{lemma}[thm]{Lemma}
\newtheorem{corollary}[thm]{Corollary}
\newtheorem{proposition}[thm]{Proposition}
\newtheorem{addendum}[thm]{Addendum}
\newtheorem{variant}[thm]{Variant}
\theoremstyle{definition}
\newtheorem{lemma and definition}[thm]{Lemma and Definition}
\newtheorem{construction}[thm]{Construction}
\newtheorem{statement}[thm]{Statement}
\newtheorem{notations}[thm]{Notations}
\newtheorem{question}[thm]{Question}
\newtheorem{problem}[thm]{Problem}
\newtheorem{remark}[thm]{Remark}
\newtheorem{remarks}[thm]{Remarks}
\newtheorem{definition}[thm]{Definition}
\newtheorem{claim}[thm]{Claim}
\newtheorem{assumption}[thm]{Assumption}
\newtheorem{assumptions}[thm]{Assumptions}
\newtheorem{properties}[thm]{Properties}
\newtheorem{example}[thm]{Example}
\newtheorem{conjecture}[thm]{Conjecture}
\newtheorem{proposition and definition}[thm]{Proposition and Definition}
\numberwithin{equation}{thm}
\newcommand{\Spec}{\mathrm{Spec}}
\newcommand{\pP}{{\mathfrak p}}
\newcommand{\sA}{{\mathcal A}}
\newcommand{\sB}{{\mathcal B}}
\newcommand{\sC}{{\mathcal C}}
\newcommand{\sD}{{\mathcal D}}
\newcommand{\sE}{{\mathcal E}}
\newcommand{\sF}{{\mathcal F}}
\newcommand{\sG}{{\mathcal G}}
\newcommand{\sH}{{\mathcal H}}
\newcommand{\sI}{{\mathcal I}}
\newcommand{\sJ}{{\mathcal J}}
\newcommand{\sK}{{\mathcal K}}
\newcommand{\sL}{{\mathcal L}}
\newcommand{\sM}{{\mathcal M}}
\newcommand{\sN}{{\mathcal N}}
\newcommand{\sO}{{\mathcal O}}
\newcommand{\sP}{{\mathcal P}}
\newcommand{\sQ}{{\mathcal Q}}
\newcommand{\sR}{{\mathcal R}}
\newcommand{\sS}{{\mathcal S}}
\newcommand{\sT}{{\mathcal T}}
\newcommand{\sU}{{\mathcal U}}
\newcommand{\sV}{{\mathcal V}}
\newcommand{\sW}{{\mathcal W}}
\newcommand{\sX}{{\mathcal X}}
\newcommand{\sY}{{\mathcal Y}}
\newcommand{\sZ}{{\mathcal Z}}
\newcommand{\A}{{\mathbb A}}
\newcommand{\B}{{\mathbb B}}
\newcommand{\C}{{\mathbb C}}
\newcommand{\D}{{\mathbb D}}
\newcommand{\E}{{\mathbb E}}
\newcommand{\F}{{\mathbb F}}
\newcommand{\G}{{\mathbb G}}
\renewcommand{\H}{{\mathbb H}}
\newcommand{\I}{{\mathbb I}}
\newcommand{\J}{{\mathbb J}}
\renewcommand{\L}{{\mathbb L}}
\newcommand{\M}{{\mathbb M}}
\newcommand{\N}{{\mathbb N}}
\renewcommand{\P}{{\mathbb P}}
\newcommand{\Q}{{\mathbb Q}}
\newcommand{\Qbar}{\overline{\Q}}
\newcommand{\R}{{\mathbb R}}
\newcommand{\SSS}{{\mathbb S}}
\newcommand{\T}{{\mathbb T}}
\newcommand{\U}{{\mathbb U}}
\newcommand{\V}{{\mathbb V}}
\newcommand{\W}{{\mathbb W}}
\newcommand{\Z}{{\mathbb Z}}
\newcommand{\g}{{\gamma}}
\newcommand{\id}{{\rm id}}
\newcommand{\rk}{{\rm rank}}
\newcommand{\END}{{\mathbb E}{\rm nd}}
\newcommand{\End}{{\rm End}}
\newcommand{\Hom}{{\rm Hom}}
\newcommand{\Hg}{{\rm Hg}}
\newcommand{\tr}{{\rm tr}}
\newcommand{\Sl}{{\rm Sl}}
\newcommand{\GL}{{\rm Gl}}
\newcommand{\Cor}{{\rm Cor}}
\newcommand{\HIG}{\mathrm{HIG}}
\newcommand{\MIC}{\mathrm{MIC}}
\newcommand{\SO}{{\rm SO}}
\newcommand{\OO}{{\rm O}}
\newcommand{\SP}{{\rm SP}}
\newcommand{\Sp}{{\rm Sp}}
\newcommand{\UU}{{\rm U}}
\newcommand{\SU}{{\rm SU}}
\newcommand{\SL}{{\rm SL}}

\newcommand{\ra}{\rightarrow}
\newcommand{\xra}{\xrightarrow}
\newcommand{\la}{\leftarrow}
\newcommand{\Nm}{\mathrm{Nm}}
\newcommand{\Gal}{\mathrm{Gal}}
\newcommand{\Res}{\mathrm{Res}}
\newcommand{\Gl}{\mathrm{Gl}}
\newcommand{\Gr}{\mathrm{Gr}}
\newcommand{\Exp}{\mathrm{Exp}}
\newcommand{\Sym}{\mathrm{Sym}}
\newcommand{\Aut}{\mathrm{Aut}}
\newcommand{\GSp}{\mathrm{GSp}}
\newcommand{\Tr}{\mathrm{Tr}}
\newcommand{\TP}{\mathrm{TP}}
\newcommand{\Ext}{\mathrm{Ext}}
\newcommand{\bA}{\mathbf{A}}
\newcommand{\bK}{\mathbf{K}}
\newcommand{\bM}{\mathbf{M}} 
\newcommand{\bP}{\mathbf{P}}
\newcommand{\bC}{\mathbf{C}}
\maketitle

\begin{abstract}
We establish the twisted functoriality in nonabelian Hodge theory in positive characteristic. As an application, we obtain a purely algebraic proof of the fact that the pullback of a semistable Higgs bundle with vanishing Chern classes is again semistable.
\end{abstract}
\section{Introduction}
In the classical nonabelian Hodge theory \cite{Sim92}, one has the following Simpson correspondence: Let $X$ be a compact K\"{a}hler manifold. There is an equivalence of categories
$$
C^{-1}_{X}: \HIG(X)\to \MIC(X),
$$
where $\HIG(X)$ is the category of polystable Higgs bundles over $X$ with vanishing first two Chern classes and $\MIC(X)$ is the category of semisimple flat bundles over $X$. The equivalence is independent of the choice of a background K\"{a}hler metric, and the following functoriality holds: Let $f: Y\to X$ be a morphism of compact K\"{a}hler manifolds. Then for any $(E,\theta)\in \HIG(X)$, one has a natural isomorphism in $\MIC(Y)$
\begin{equation}\label{functoriality over C}
C_{Y}^{-1}f^*(E,\theta)\cong f^*C^{-1}_{X}(E,\theta).
\end{equation}
In the nonabelian Hodge theory in positive characteristic \cite{OV}, Ogus-Vologodsky established an analogue of \ref{functoriality over C} for derived categories, with the $W_2$-lifting assumption on $f$ (see Theorem 3.22 \cite{OV}). In a recent preprint \cite{La2}, A. Langer proved the equality \ref{functoriality over C} under the assumption that the $W_2(k)$-lifting of $f$ is good (see Definition 5.1 and Theorem 5.3 in loc. cit.). However, such an assumption on the lifting of $f$ is quite restrictive.

Let $k$ be a perfect field of characteristic $p>0$. Let $X$ be a smooth variety over $k$ and $D$ a reduced normal crossing divisor in $X$. One forms the log smooth variety $X_{\log}$ whose log structure is the one determined by $D$. Equip $k$ and $W_2(k)$ with the trivial log structure. Assume that the log morphism $X_{\log}\to k$ is liftable to $W_2(k)$. Choose and then fix such a lifting $\tilde X_{\log}$. Then one has the inverse Cartier transform \footnote{Theorem 6.1 \cite{LSYZ19} deals with only the case of SNCD. However, a simple \'{e}tale descent argument extends the construction to the reduced NCD case.} (which is in general not an equivalence of categories without further condition on the singularities of modules along $D$)
$$
C^{-1}_{X_{\log}\subset \tilde X_{\log}}: \HIG_{\leq p-1}(X_{\log}/k)\to \MIC_{\leq p-1}(X_{\log}/k).
$$
Let $Y_{\log}=(Y,B)$ be a log smooth variety like above, together with a $W_2(k)$-lifting $\tilde Y_{\log}$. Our main result is the following analogue of \ref{functoriality over C} in positive characteristic:
\begin{theorem}\label{main result}
Notion as above. Then for any object $(E,\theta)\in \HIG_{\leq p-1}(X_{\log}/k)$, one has a natural isomorphism
$$
C_{Y_{\log}\subset \tilde Y_{\log}}^{-1}f^{\circ}(E,\theta)\cong f^*C^{-1}_{X_{\log}\subset \tilde X_{\log}}(E,\theta),
$$
where $f^{\circ}(E,\theta)$ is the twisted pullback of $(E,\theta)$.
\end{theorem}
The twisted pullback of $(E,\theta)$ refers to a certain deformation of $f^*(E,\theta)$ along the obstruction class of lifting $f$ over $W_2(k)$. When the obstruction class vanishes, the twisted pullback is just the usual pullback. See \S2 for details. Hence, one has the following immediate consequence. 
\begin{corollary}\label{main cor}
Let $f: Y_{\log}\to X_{\log}$ be a morphism of log smooth varieties over $k$. Assume $f$ is liftable to $W_2(k)$. Then for any object $(E,\theta)\in \HIG_{\leq p-1}(X_{\log}/k)$, one has a natural isomorphism in $\MIC_{\leq p-1}(Y'_{\log}/k)$
$$
C_{Y_{\log}\subset \tilde Y_{\log}}^{-1}f^{*}(E,\theta)\cong f^*C^{-1}_{X_{\log}\subset \tilde X_{\log}}(E,\theta).
$$
\end{corollary}
The notion of \emph{twisted pullback} and the corresponding \emph{twisted functoriality} as exhibited in Theorem \ref{main result} was inspired by the work of Faltings in the $p$-adic Simpson correspondence \cite{Fa05}. It is a remarkable fact that char $p$ and $p$-adic Simpson correspondences have many features in common. As an application, we obtain the following result.
\begin{theorem}\label{semistability}
Let $k$ be an algebraically closed field and $f: (Y,B)\to (X,D)$ a morphism between smooth projective varieties equipped with normal crossing divisors over $k$. Let $(E,\theta)$ be a semistable logarithmic Higgs bundles with vanishing Chern classes over $(X,D)$. If either $\textrm{char}(k)=0$ or $\textrm{char}(k)=p>0$, $f$ is $W_2(k)$-liftable and $\rk(E)\leq p$, then the logarithmic Higgs bundle $f^*(E,\theta)$ over $(Y,B)$ is also semistable with vanishing Chern classes. 
 \end{theorem}
For $\textrm{char}(k)=0$ and $D=\emptyset$, the result is due to C. Simpson by transcendental means \cite{Sim92}. Our approach is to deduce it from the char $p$ statement by mod $p$ reduction and hence is purely algebraic.

\section{Twisted pullback}
We assume our schemes are all noetherian. Let $(R,M)$ be an affine log scheme. Let $f: Y\to X$ be a morphism of log smooth schemes over $R$. Fix an $r\in \N$. Choose and then fix an element $\tau\in \Ext^1(f^*\Omega_{X/R},\sO_Y)$. The aim of this section is to define the twisted pullback along $\tau$ as a functor
$$
\TP_{\tau}: \HIG_{\leq r}(X/R)\to \HIG_{\leq r}(Y/R),
$$
under the following assumption on $r$
\begin{assumption}\label{basic assumption on r}
	$r!$ is invertible in $R$.
\end{assumption}

Let $\Omega_{X/R}$ be the sheaf of relative logarithmic K\"{a}hler differentials and $T_{X/R}$ be its $\sO_X$-dual. They are locally free of rank $\dim X-\dim R$ by log smoothness. The symmetric algebra $\Sym^{\bullet}T_{X/R}=\bigoplus_{k\geq 0} \Sym^kT_{X/R}$ on $T_{X/R}$ is $\sO_X$-algebra, and one has the following morphisms of $\sO_X$-algebras whose composite is the identity:
$$
\sO_X\to \Sym^{\bullet}T_{X/R}\to \sO_X.
$$
It defines the zero section of the natural projection $\Omega_{X/R}\to X$, where we view $\Omega_{X/R}$ as a vector bundle over $X$ (see Ex 5.18, Ch. II \cite{Ha77}). Set $$\sA_r:=\Sym^{\bullet}(T_{X/R})/\Sym^{\geq r+1}(T_{X/R}),$$ which is nothing but the structure sheaf of the closed subscheme $(r+1)X$ of $\Omega_{X/R}$ supported along the zero section. In below, we shall use the notations $\sA_r$ and $\sO_{(r+1)X}$ interchangeably. Note as $\sO_X$-module, $\sA_{r}=\sO_X\oplus T_{X/R}\oplus \cdots \oplus \Sym^{r}T_{X/R}$. The following lemma is well-known.
\begin{lemma}\label{correspondence}
The category of nilpotent (quasi-)coherent Higgs modules over $X/R$ of exponent $\leq r$ is equivalent to the category of (quasi-)coherent $\sO_{(r+1)X}$-modules.
\end{lemma}
\begin{proof}
The natural inclusion $\iota: X\to \Omega_{X/R}$ of zero section induces an equivalence of categories between the category of sheaves of abelian groups over $X$ and the category of sheaves of abelian groups over $\Omega_{X/R}$ whose support is contained in the zero section. Let $E$ be a sheaf of abelian groups over $X$. It has a Higgs module structure if it has
\begin{itemize}
\item [(i)] a ring homomorphism $\theta^0:\sO_X\to \End(E)$;
\item [(ii)] an $\sO_X$-linear homomorphism $\theta^1: T_{X/R}\to \End_{\sO_X}(E)$.
\end{itemize}
Since $\Sym^{\bullet}T_{X/R}$ is generated by $T_{X/R}$ as $\sO_X$-algebra, $\theta^0$ and $\theta^1$ together extend to a ring homomorphism
$$
\theta^{\bullet}: \Sym^{\bullet}T_{X/R}\to \End_{\sO_X}(E)\subset \End(E).
$$
If $\theta^1$ is nilpotent of exponent $\leq r$, then $\Sym^{\geq r+1}(T_{X/R})\subset \mathrm{Ann}(E)$. Therefore, we obtain an $\sA_r$-module structure on $E$. So we obtain a sheaf of $\sO_{(r+1)X}$-module. As $E$ is (quasi-)coherent as $\sO_X$-module, it is (quasi-)coherent as $\sO_{(r+1)X}$-module. Conversely, for a quasi-coherent $\sO_{\sA_r}$-module $E$, one obtains a ring homomorphism
$$
\sA_r\to \End(E).
$$
Restricting it to the degree zero part, one obtains the $\sO_X$-module structure on $E$. While restricting to the degree one component, one obtains a morphism of sheaf of abelian groups
$$
\theta: T_{X/R}\to \End(E), v\mapsto \theta_v:=\textrm{the multiplication by}\ v.
$$
Since for any $v\in T_{X/R}$, $v^{r+1}=0$ in $\sA_r$, it follows $\theta^{r+1}=0$, that is the exponent of $\theta\leq r$. For any $f\in\sO_X, v\in T_{X/R}$ and any $e\in E$, one verifies that
$$
\theta_v(fe)=\theta_{fv}(e)=f\theta_v(e),
$$
which means that the image of $\theta$ is contained in $\End_{\sO_X}(E)$. The obtained $\sO_X$-module is nothing but the pushforward of $E$ along the composite $(r+1)X\nrightarrow \Omega_{X/R}\to X$ which is finite. Therefore, $E$ is (quasi-)coherent as $\sO_X$-module if it is (quasi-)coherent as $\sO_{(r+1)X}$-module.
\end{proof}
\begin{remark}\label{equivalence between A-module and Higgs module}
An $f^*\Omega_{X/R}$-Higgs module is a pair $(E,\theta)$ where $E$ is an $\sO_Y$-module and $\theta: E\to E\otimes f^*\Omega_{X/R}$ is an $\sO_Y$-linear morphism satisfying $\theta\wedge \theta=0$. A modification of the above argument shows that the category of nilpotent (quasi-)coherent $f^*\Omega_{X/R}$-Higgs modules is equivalent to the category of (quasi-)coherent $f^*\sA_r$-modules.
\end{remark}

{\bf $1^{st}$ construction:} For an $r$ satisfying Assumption \ref{basic assumption on r}, we have a natural morphism:
$$
\exp: H^1(Y,f^*T_{X/R})\to H^1(Y, (f^*\sA_r)^*), \tau\mapsto \exp(\tau)=1+\tau+\cdots+\frac{\tau^r}{r!},
$$
where $(f^*\sA_r)^*$ is the unit group of $f^*\sA_r$. An element of $f^*\sA_r$ is invertible iff its image under $f^*\sA_r\to \sO_{Y}$ is invertible. So we obtain an $f^*\sA_r$-module $\sF^{r}_{\tau}$ of rank one.
We introduce an intermediate category $\HIG_{\leq r}(f^*\Omega_{X/R})$, which is the category of nilpotent quasi-coherent $f^*\Omega_{X/R}$-Higgs modules of exponent $\leq r$. We define the functor
$$
\TP^{\sF}_{\tau}: \HIG_{\leq r}(X/R)\to \HIG_{\leq r}(f^*\Omega_{X/R})
$$
as follows: For an $E\in \HIG_{\leq r}(X/R)$, define 
$$
\TP^{\sF}_{\tau}(E):=\sF^r_{\tau}\otimes_{f^*\sA_r}f^{*}E
$$
as $f^*\sA_r$-module. Next, for a morphism $\phi: E_1\to E_2$ in $\HIG_{\leq r}(X/R)$, $$\TP^{\sF}_{\tau}(\phi):=id\otimes f^*\phi: \TP^{\sF}_{\tau}(E_1)\to \TP^{\sF}_{\tau}(E_2)$$ is a morphism of $f^*\sA_r$-modules. One has the natural functor from $\HIG_{\leq r}(f^*\Omega_{X/R})$ to $\HIG_{\leq r}(\Omega_{Y/R})$ induced by the differential morphism $f^*\Omega_{X/R}\to \Omega_{Y/R}$. We define the functor $\TP^1_{\tau}$ to be composite of functors
$$
\HIG_{\leq r}(X/R)\stackrel{\TP^{\sF}_{\tau}}{\longrightarrow}\HIG_{\leq r}(f^*\Omega_{X/R})\to \HIG_{\leq r}(Y/R).
$$
This is how Faltings \cite{Fa05} defines twisted pullback in the $p$-adic setting, at least for those small $\tau$s. \\

{\bf $2^{nd}$ construction:} This is based on the method of \emph{exponential twisting} \cite{LSZ15}, whose basic construction is given as follows: 

{\itshape Step 0}: Take an open affine covering $\{U_{\alpha}\}_{\alpha\in \Lambda}$ of $X$ as well as an open affine covering $\{V_{\alpha}\}_{\alpha\in \Lambda}$ of $Y$ such that $f: V_{\alpha}\to U_{\alpha}$.    Let $\{\tau_{\alpha\beta}\}$ be a Cech representative of $\tau$. That is, $\tau_{\alpha\beta}\in \Gamma(V_{\alpha\beta}, f^*T_{X/R})$ satisfying the cocycle relation
$$
\tau_{\alpha\gamma}=\tau_{\alpha\beta}+\tau_{\beta\gamma}.
$$
{\itshape Step 1}: Let $(E,\theta)$ be a nilpotent Higgs module over $X$, whose exponent of nilpotency satisfies Assumption \ref{basic assumption on r}. For any $\alpha$, set $(E_{\alpha},\theta_{\alpha})=(E,\theta)|_{U_{\alpha}}$. Then one forms the various local Higgs modules $\{(f^*E_{\alpha},f^*\theta_{\alpha})\}$ via the usual pullback. \\

{\itshape Step 2}: Define
$$
G_{\alpha\beta}=\exp(\tau_{\alpha\beta}\cdot f^*\theta)=\sum_{i\geq 0}\frac{(\tau_{\alpha\beta}\cdot f^*\theta)^n}{n!}.
$$
The expression makes sense since each term $\frac{(\tau_{\alpha\beta}\cdot f^*\theta)^n}{n!}$ is well defined by assumption. Obviously, $G_{\alpha\beta}\in \Aut_{\sO_{Y}}(f^*E|_{V_{\alpha\beta}})$. Because of the cocycle relation, $\{G_{\alpha\beta}\}$ satisfies the cocycle relation
$$
G_{\alpha\gamma}=G_{\beta\gamma}G_{\alpha\beta}.
$$
Then we use the set of local isomorphism $\{G_{\alpha\beta}\}$ to glue the local $\Omega_{Y/R}$-Higgs modules $\{(f^*E_{\alpha},f^*\theta_{\alpha})\}$, to obtain a new Higgs module over $Y$. The verification details are analogous to \S2.2 \cite{LSZ15}. It is tedious and routine to verify the glued Higgs module, up to natural isomorphism, is independent of the choice of affine coverings and Cech representatives of $\tau$. We denote it by $\TP^2_{\tau}(E)$. For a morphism $\phi: E_1\to E_2$ of Higgs modules, it is not difficult to see that $f^*\phi$ induces a morphism $\TP^2_{\tau}(\phi):\TP^2_{\tau}(E_1)\to \TP^2_{\tau}(E_2)$. 
\begin{proposition}\label{Faltings twisted pullback is exponential twisting}
The two functors $\TP^1_{\tau}$ and $\TP^2_{\tau}$ are naturally isomorphic. 
\end{proposition}
\begin{proof}
One uses the equivalence in Remark \ref{equivalence between A-module and Higgs module}. It suffices to notice that the element $\exp(\tau_{\alpha\beta})\in f^*\sA_r$ has its image $G_{\alpha\beta}$ in $\Aut_{\sO_{Y}}(f^*E|_{V_{\alpha\beta}})$.
\end{proof}
By the above proposition, we set $\TP_{\tau}$ to be either of $\TP^i_{\tau}, i=1.2$.
\begin{proposition}
The functor $\TP_{\tau}$ has the following properties:
\begin{itemize}
	\item [(i)] it preserves rank;
	\item [(ii)] it preserves direct sum;
	\item [(iii)] Let $E_i, i=1,2$ be two nilpotent Higgs modules over $X/R$ whose exponents of nilpotency satisfies $(r_1+r_2)!$ being invertible in $R$. Then there is a canonical isomorphism of Higgs modules over $Y$:
	$$
	\TP_{\tau}(E_1\otimes E_2) \cong \TP_{\tau}(E_1)\otimes\TP_{\tau}(E_2).
	$$
\end{itemize} 
\end{proposition}
\begin{proof}
The first two properties are obvious. To approach (iii), one uses the second construction. Note that when $\tau=0$, it is nothing but the fact $f^*(E_1\otimes E_2)=f^*E_1\otimes f^*E_2$. When the exponents $r_i, i=1,2$ satisfies the condition, one computes that
	$$
	\exp(\tau\cdot (f^*\theta_1\otimes id+id\otimes f^*\theta_2))=\exp(\tau\cdot f^*\theta_1\otimes id)\exp(id\otimes\tau\cdot f^*\theta_2),
	$$
using the equality
	$$
	(f^*\theta_1\otimes id)(id\otimes f^*\theta_2)=(id\otimes f^*\theta_2)(f^*\theta_1\otimes id)=f^*\theta_1\otimes f^*\theta_2.
	$$
\end{proof}

To conclude this section, we shall point out that there is one closely related construction that works for all $r\in \N$. \\
 
{\bf $3^{rd}$ construction:} Note that the element $\tau\in \Ext^1(f^*\Omega_X,\Omega_Y)\cong \Ext^1(\sO_Y,f^*T_{X/R})$ corresponds to an extension of $\sO_Y$-modules
$$
0\to f^*T_{X/R}\to \sE_{\tau}\stackrel{pr}{\to} \sO_{Y} \to 0.
$$
Notice that $\sE_{\tau}$ admits a natural $f^*\Sym^{\bullet}T_{X/R}$-module structure: In degree zero, this is $\sO_{Y}$-structure; in degree one,
$$
f^*T_{X/R}\otimes_{\sO_{Y}}\sE_{\tau}\stackrel{id\otimes pr}{\longrightarrow} f^*T_{X/R}\otimes_{\sO_{Y}}\sO_{Y}=f^*T_{X/R}\subset \sE_{\tau},
$$
and therefore $\theta: f^*T_{X/R}\to \End_{\sO_{Y}}(\sE_{\tau})$. By construction, $\theta\neq 0$ but $\theta^2=0$. For any $r\in \N$, set
$$
\sE^r_{\tau}:=\Sym^r\sE_{\tau}.
$$
The proof of the next lemma is straightforward.
\begin{lemma}\label{basic property of E^r}
	For any $r\in \N$, $\sE^r_{\tau}$ is a nilpotent $f^*\Omega_{X/R}$-Higgs bundle of exponent $r$. It admits a filtration $F^{\bullet}$ of $f^*\Omega_{X/R}$-Higgs subbundles:
	$$
	\sE^r_{\tau}=F^0\supset F^1\supset\cdots \supset F^r\supset 0,
	$$
	whose associated graded $Gr_{F^{\bullet}}E^r_{\tau}$ is naturally isomorphic to $f^*\sA_r$. When $\tau=0$, $\sE^r_{\tau}=f^*\sA_r$ as $f^*\Omega_{X/R}$-Higgs bundle.
\end{lemma}
By the lemma, $\sE^r_{\tau}$ is an $f^*\sA_r$-module of rank one. Therefore, one may replace the tensor module in the definition of $\TP^{\sF}_{\tau}$ with $\sE^{r}_{\tau}$. This defines a new functor $\TP^{\sE}_{\tau}$ and hence the third twisted pullback functor $\TP^3_{\tau}$.

\begin{remark}
	When one is interested only in coherent objects, one may drop the nilpotent condition in the construction. This is because by Cayley-Hamilton, there is an element of form
	$v^r-a_1v^{r-1}+\cdots+(-1)^ra_r\in \Sym^{\bullet}T_{X/R}$ annihilating $E$, so that $\Sym^{\bullet}T_{X/R}$-module structure on $E$ factors though $\Sym^{\bullet}T_{X/R}\to \sA_r$.
\end{remark}
We record the following statement for further study. 
\begin{proposition}
Assume $r\in \N$ satisfy Assumption \ref{basic assumption on r}. Then as $f^*\sA_r$-modules,
\begin{itemize}
	\item [(i)] $\sF^r_{\tau}\cong \sE^r_{\tau}$ for $r\leq 1$;
	\item [(ii)] $\sF^r_{\tau}\ncong \sE^r_{\tau}$ for $r>1$.
\end{itemize}
\end{proposition}
\begin{proof}
Obviously, $\sF^0_{\tau}\cong \sE^0_{\tau}\cong \sO_Y$. Assume $r\geq 1$. We illustrate our proof by looking at the case of $X/R$ being a relative curve. We describe $\sE^r_{\tau}$ in terms of local data: Take $r=1$ first. Let $U_{\alpha}$ be an open subset of $X$ with $\partial_{\alpha}$ a local basis of $\Gamma(U_{\alpha}, T_{X/R})$. Assume that $V_{\alpha}$ to be an open subset of $Y$ such that $f: V_{\alpha}\to U_{\alpha}$. We may assume the gluing functions between two different local basis are identity. Let $\{\tau_{\alpha\beta}\}$ be a Cech representative of $\tau$. Write $\tau_{\alpha\beta}=a_{\alpha\beta}f^*\partial_{\alpha\beta}$. Then $\sE_{\tau}$ is the $\sO_{Y}$-module obtained by gluing $\{\sO_{V_{\alpha}}\oplus f^*T_{U_{\alpha}/R}\}$ via the following gluing matrix:
$$
 \left(
   \begin{array}{c}
     1 \\
     f^*\partial_{\alpha} \\
   \end{array}
 \right)=\left(
           \begin{array}{cc}
             1 & a_{\alpha\beta} \\
             0 & 1 \\
           \end{array}
         \right)\left(
                  \begin{array}{c}
                    1 \\
                      f^*\partial_{\beta} \\
                  \end{array}
                \right).
$$
Under the assumption for $r$, $\sE^r_{\tau}$ is obtained by gluing $$\{f^*\sA_{r}|_{U_{\alpha}}=\sO_{V_{\alpha}}\oplus f^*T_{U_{\alpha}/R}\oplus\cdots \oplus f^*T^{\otimes r}_{U_{\alpha}/R}\}$$ via the gluing matrix:
$$
\left(
  \begin{array}{c}
    \frac{1}{r!} \\
      \frac{f^*\partial_{\alpha}}{(r-1)!}  \\
    \vdots\\
     f^*\partial^{r-1}_{\alpha} \\
   f^*\partial^r_{\alpha} \\
  \end{array}
\right)=\left(
          \begin{array}{ccccc}
            1& a_{\alpha\beta} & \frac{a^2_{\alpha\beta}}{2!} & \ldots & \frac{a^r_{\alpha\beta}}{r!} \\
            0& 1 & a_{\alpha\beta} & \ldots & \frac{a^{r-1}_{\alpha\beta}}{(r-1)!} \\
            \vdots &  \vdots &  \ddots &  \ddots &  \vdots \\
            0 & 0 & \ldots & 1 & a_{\alpha\beta} \\
          0 & 0 & \ldots & \ldots & 1 \\
          \end{array}
        \right)\left(
  \begin{array}{c}
    \frac{1}{r!} \\
      \frac{f^*\partial_{\beta}}{(r-1)!}  \\
    \vdots\\
     f^*\partial^{r-1}_{\beta} \\
   f^*\partial^r_{\beta} \\
  \end{array}
\right).
$$
As comparison, $\sF^r_{\tau}$ is obtained by gluing $\{f^*\sA_{r}|_{U_{\alpha}}\}$ via the following transition functions
$$
\left(
\begin{array}{c}
1 \\
f^*\partial_{\alpha}  \\
\vdots\\
f^*\partial^{r-1}_{\alpha} \\
f^*\partial^r_{\alpha} \\
\end{array}
\right)=\left(
\begin{array}{ccccc}
1& a_{\alpha\beta} & \frac{a^2_{\alpha\beta}}{2!} & \ldots & \frac{a^r_{\alpha\beta}}{r!} \\
0& 1 & a_{\alpha\beta} & \ldots & \frac{a^{r-1}_{\alpha\beta}}{(r-1)!} \\
\vdots &  \vdots &  \ddots &  \ddots &  \vdots \\
0 & 0 & \ldots & 1 & a_{\alpha\beta} \\
0 & 0 & \ldots & \ldots & 1 \\
\end{array}
\right)\left(
\begin{array}{c}
1 \\
f^*\partial_{\beta}  \\
\vdots\\
f^*\partial^{r-1}_{\beta} \\
f^*\partial^r_{\beta} \\
\end{array}
\right).
$$
Therefore, $\sF^r_{\tau}$ and $\sE^{r}_{\tau}$ are isomorphic as $\sO_Y$-modules. However, when $r\geq 2$, the Higgs structures of these two bundles differ: For $\sF^r_{\tau}$, the Higgs field along $\partial_{\alpha}$ is given by 
$$
\theta_{\partial_\alpha}\left(
\begin{array}{c}
1 \\
f^*\partial_{\alpha}  \\
\vdots\\
f^*\partial^{r-1}_{\alpha} \\
f^*\partial^r_{\alpha} \\
\end{array}
\right)=\left(
\begin{array}{ccccc}
0& 1 & 0 & \ldots & 0 \\
0& 0 & 1 & \ldots & 0 \\
\vdots &  \vdots &  \ddots &  \ddots &  \vdots \\
0 & 0 & \ldots & 0 & 1 \\
0 & 0 & \ldots & \ldots & 0 \\
\end{array}
\right)\left(
\begin{array}{c}
1 \\
f^*\partial_{\alpha}  \\
\vdots\\
f^*\partial^{r-1}_{\alpha} \\
f^*\partial^r_{\alpha} \\
\end{array}
\right),
$$
while for Higgs field action for $\sE^r_{\tau}$ is given by
$$
\theta_{\partial_\alpha}\left(
\begin{array}{c}
\frac{1}{r!} \\
\frac{f^*\partial_{\alpha}}{(r-1)!}  \\
\vdots\\
f^*\partial^{r-1}_{\alpha} \\
f^*\partial^r_{\alpha} \\
\end{array}
\right)=\left(
\begin{array}{ccccc}
0& \frac{1}{r} & 0 & \ldots & 0 \\
0& 0 & \frac{1}{r-1} & \ldots & 0 \\
\vdots &  \vdots &  \ddots &  \ddots &  \vdots \\
0 & 0 & \ldots & 0& 1 \\
0 & 0 & \ldots & \ldots & 0 \\
\end{array}
\right)\left(
\begin{array}{c}
\frac{1}{r!} \\
\frac{f^*\partial_{\alpha}}{(r-1)!}  \\
\vdots\\
f^*\partial^{r-1}_{\alpha} \\
f^*\partial^r_{\alpha} \\
\end{array}
\right).
$$

\end{proof}

\section{Twisted functoriality}
Now we come back to the setting in \S1. First we make the following
\begin{definition}\label{twisted pullback}
Let $k$, $f: Y_{\log}\to X_{\log}$ and $\tilde X_{\log}, \tilde Y_{\log}$ be as in \S1. For a Higgs module $(E,\theta)\in \HIG_{\leq p-1}(X_{\log}/k)$.  Then the twisted pullback $f^\circ(E,\theta)$ is defined to be $\TP_{ob(f)}(E,\theta)$, where $ob(f)$ is the obstruction class of lifting $f$ to a morphism $\tilde Y_{\log}\to \tilde X_{\log}$ over $W_2(k)$.
\end{definition}
Assume that $f$ admits a $W_2(k)$-lifting $\tilde f$. In Langer's proof of functoriality Theorem 5.3 \cite{La19}, the existence of local logarithmic Frobenius liftings $F_{\tilde X_{\log}}$ and $F_{\tilde Y_{\log}}$ such that $F_{\tilde X_{\log}}\circ \tilde f=\tilde f\circ F_{\tilde Y_{\log}}$ is crucial-this is where the condition of $\tilde f$ being good enters. However, one notices that any local logarithmic Frobenius liftings on $\tilde X$ and $\tilde Y$ commute with $\tilde f$ \emph{up to homotopy}. A heuristic reasoning shows that this homotopy should be intertwined with the homotopies caused by local logarithmic Frobenius liftings of both $\tilde X$ and $\tilde Y$, as well as the one caused by local liftings of the morphism (no $W_2$-lifting on $f$ is assumed any more). Turning this soft homotopy argument into exact differential calculus in positive characteristic yields the proof for the claimed twisted functoriality.

To start with proof of Theorem \ref{main result}, we take an \'{e}tale covering $\mathcal X=\coprod_i X_i\to X$ with $X_i$ affine and the pullback of $D$ along each $X_i\to X$ simple normal crossing. Then we take an \'{e}tale covering $\pi: \mathcal Y=\coprod_i Y_i\to Y$ with similar properties and $f$ restricts to a local morphism $f_i: Y_{i,\log}\to X_{i,\log}$ for each $i$. For each $i$, we choose logarithmic Frobenius lifting over $W_2(k)$
$$
F_{\tilde X_{i,\log}}: \tilde X_{i,\log}\to \tilde X_{i,\log},\quad F_{\tilde Y_{i,\log}}: \tilde Y_{i,\log}\to \tilde Y_{i,\log},
$$
and also a $W_2(k)$-lift $\tilde f_i: \tilde Y_{i,\log}\to \tilde X_{i,\log}$. Such local lifts exist. Set
$$
(V_1,\nabla_1)=C_{Y_{\log}\subset \tilde Y_{\log}}^{-1}f^{\circ}(E,\theta),\quad (V_2,\nabla_2)=f^*C^{-1}_{X_{\log}\subset \tilde X_{\log}}(E,\theta).
$$
In below, we exhibit an isomorphism between $(V_i,\nabla_i),i=1,2$ after pulling back to the \'{e}tale covering $\sY$ which satisfies the descent condition. The whole proof is therefore divided into two steps.\\

{\itshape Step 1: Isomorphism over $\sY$}\\
As $\sY$ is a disjoint union of open affine log schemes $\{Y_{i,\log}\}$s, it suffices to construct an isomorphism for each open affine. In the foregoing argument, we drop out the subscript $i$ everywhere. Notice first that the two morphisms $\tilde f^*\circ F_{\tilde X_{\log}}^*$ and $F_{\tilde Y_{\log}}^*\circ \tilde f^*$ coincide after reduction modulo $p$. Thus, it defines an element
$$
\nu_{f}\in \Hom_{\sO_{X}}(\Omega_{X_{\log}/k},\sO_Y)
$$
such that
$$
\nu_f\circ d=\frac{1}{p}(F_{\tilde Y_{\log}}^*\circ \tilde f^*-\tilde f^*\circ F_{\tilde X_{\log}}^*).
$$
So we get $\nu_f\cdot \theta\in \Gamma(Y,\End_{\sO_Y}(g^*E))$, where $g=F_X\circ f=f\circ F_Y$.
\begin{lemma}
$\exp(\nu_f\cdot \theta)$ defines an isomorphism $(V_1,\nabla_1)\to (V_2,\nabla_2)$. That is, there is a commutative diagram:
$$\CD
  V_1 @> \exp(\nu_f\cdot \theta) >> V_2 \\
  @V \nabla_1 VV @VV  \nabla_2V  \\
  V_1\otimes \Omega_{Y_{\log}/k} @>\exp(\nu_f\cdot \theta)\otimes id>> V_2\otimes \Omega_{Y_{\log}/k}.
\endCD$$
\end{lemma}
\begin{proof}
Recall that over $Y$, $V_1=V_2=g^*E$. So $\exp(\nu_f\cdot E)$ defines an isomorphism from $V_1$ to $V_2$. Moreover, the connections are given by
$$
\nabla_1=\nabla_{can}+(id\otimes \frac{dF_{\tilde Y_{\log}}}{p})(F_Y^*f^*\theta),
$$
and respectively by
$$
\nabla_2=f^*(\nabla_{can}+(id\otimes \frac{dF_{\tilde X_{\log}}}{p})(F_X^*\theta)).
$$
Now we are going to check the commutativity of the above diagram. Take a local section $e\in E$. Then
$$
\exp(\nu_f\cdot \theta)\otimes id\circ \nabla_1(g^*e)= \exp(\nu_f\cdot \theta)(id\otimes \frac{dF_{\tilde Y_{\log}}}{p})(F_Y^*f^*\theta(e)).
$$
On the other hand, $\nabla_2\circ \exp(\nu_f\cdot \theta)(e)$ equals
$$
\exp(\nu_f\cdot \theta)d(\nu_f\cdot \theta)(g^*e)+\exp(\nu_f\cdot \theta)(f^*(id\otimes \frac{dF_{\tilde X_{\log}}}{p})(F_X^*\theta(e))).
$$
We take a system of local coordinates $\{x_i\}$ for $\tilde X$ and use the same notion for its reduction modulo $p$. Write $\theta=\sum_i\theta_idx_i$,
and $\nu_f=\sum_iu_i\partial_{x_i}$ with $u_i\in \sO_Y$. Thus
$$
d(\nu_f\cdot \theta)=d(\sum_ig^*\theta_i\cdot u_i).
$$
As $d$ is $\sO_X$-linear, it equals
$$
\sum_ig^*\theta_i\cdot du_i=\sum_ig^*\theta_i\cdot d(\frac{(F_{\tilde Y_{\log}}^*\circ \tilde f^*-\tilde f^*\circ F_{\tilde X_{\log}}^*)(x_i)}{p}).
$$
So $d(\nu_f\cdot \theta)(g^*e)=\sum_ig^*\theta_i(e)\cdot \frac{(F_{\tilde Y_{\log}}^*\circ \tilde f^*-\tilde f^*\circ F_{\tilde X_{\log}}^*)(x_i)}{p}$. On the other hand,
\begin{eqnarray*}
(id\otimes \frac{dF_{\tilde Y_{\log}}}{p})(F_Y^*f^*\theta(e))&=& \sum_i g^*\theta_i(e)\cdot (\frac{id\otimes dF_{\tilde Y_{\log}}}{p})(F_Y^*f^*(dx_i))\\
&=&\sum_ig^*\theta_i(e)\cdot \frac{d(F_{\tilde Y_{\log}}^*\tilde f^*(x_i))}{p},
\end{eqnarray*}
and similarly,
\begin{eqnarray*}
f^*(id\otimes \frac{dF_{\tilde X_{\log}}}{p})(F_X^*\theta(e))&=& \sum_i g^*\theta_i(e)\cdot \frac{f^*d(F^*_{\tilde X_{\log}}(x_i))}{p}\\
&=&\sum_i g^*\theta_i(e)\cdot \frac{d(\tilde f^*F^*_{\tilde X_{\log}}(x_i))}{p}.
\end{eqnarray*}
This completes the proof.

\end{proof}
{\itshape Step 2: Descent condition}\\

In Step 1, we have constructed an isomorphism $\exp(\nu_f\cdot \theta): \pi^*(V_1,\nabla_1)\to \pi^*(V_2,\nabla_2)$ whose restriction to $Y_{i,\log}$ is given by $\exp(\nu_{f_i}\cdot \theta)$. Let $p_i: \sY\times_Y \sY\to \sY, i=1,2$ be two projections. In below, we show that
$$
p_1^*(\exp(\nu_f\cdot \theta))=p_2^*(\exp(\nu_f\cdot \theta)).
$$
The obstruction class $ob(F_X)$ (resp. $ob(F_Y)$ and $ob(f)$) of lifting $F_X$ (resp. $F_Y$ and $f$) over $W_2$ has its Cech representative landing in $\Gamma(X_{ij},F^*_{X}T_{X_{\log}/k})$ (resp. $\Gamma(Y_{ij},F^*_{Y}T_{Y_{\log}/k})$ and $\Gamma(Y_{ij},f^*T_{X_{\log}/k})$). We have the following natural maps:
$$
f^*: H^1(X,F_X^*T_{X_{\log}/k})\to H^1(Y,f^*F_X^*T_{X_{\log}/k})=H^1(Y,g^*T_{X_{\log}/k}),
$$
$F_Y^*: H^1(Y,f^*T_{X_{\log}/k})\to H^1(Y,g^*T_{X_{\log}/k})$, and 
$$
f_*: H^1(Y,F_Y^*T_{Y_{\log}/k})\to H^1(Y,F_Y^*f^*T_{X_{\log}/k})=H^1(Y,g^*T_{X_{\log}/k}),
$$
which is induced by $f_*: T_{Y_{\log}/k}\to f^*T_{X_{\log}/k}$.
\begin{lemma}\label{cech relation}
One has an equality in $\Gamma(Y_{ij},g^*T_{X_{\log}/k})$, where $Y_{ij}=Y_i\times_YY_j$:
$$
\nu_{f_i}-\nu_{f_j}=ob(F_Y)_{ij}+ob(f)_{ij}-ob(F_X)_{ij}
$$
where we understand the obstruction classes as their images via the natural morphisms. Consequently, there is an equality in $H^1(Y,g^*T_{X_{\log}/k})$:
$$
[\nu_{f_i}-\nu_{f_j}]=ob(F_Y)+ob(f)-ob(F_X).
$$
\end{lemma}
\begin{proof}
First, we observe the following identity
\begin{eqnarray*}
\frac{1}{p}(F^*_{\tilde Y_{i,\log}}\circ \tilde f_i^*-F^*_{\tilde Y_{j,\log}}\circ \tilde f_j^*)&=& \frac{1}{p}[(F^*_{\tilde Y_{i,\log}}-F^*_{\tilde Y_{j,\log}})\circ \tilde f_i^*]+\frac{1}{p}[F^*_{\tilde Y_{j,\log}}\circ (\tilde f_i^*-\tilde f_j^*)]\\
&=& \frac{F^*_{\tilde Y_{i,\log}}-F^*_{\tilde Y_{j,\log}}}{p}\circ f_i^*+F_{Y_{j}}^*\circ \frac{\tilde f_i^*-\tilde f_j^*}{p}.
\end{eqnarray*}
It follows that
\begin{eqnarray*}
(\nu_{f_i}-\nu_{f_j})\circ d&=&\frac{1}{p}(F^*_{\tilde Y_{i,\log}}\circ \tilde f_i^*-F^*_{\tilde Y_{j,\log}}\circ \tilde f_j^*)-\frac{1}{p}(\tilde f_i^*\circ F^*_{\tilde X_{i,\log}} -\tilde f_j^*\circ F^*_{\tilde X_{j,\log}})\\
&=& \frac{F^*_{\tilde Y_{i,\log}}-F^*_{\tilde Y_{j,\log}}}{p}\circ f_i^*+F_{Y_{j}}^*\circ \frac{\tilde f_i^*-\tilde f_j^*}{p}-f_j^*\circ\frac{F^*_{\tilde X_{i,\log}}-F^*_{\tilde X_{j,\log}}}{p}-\frac{\tilde f_i^*-\tilde f_j^*}{p}\circ F_{X_i}^*\\
&=&ob(F_Y)_{ij}\circ f_i^*\circ d+F_{Y_{j}}^*\circ ob(f)_{ij}\circ d-f_j^*\circ ob(F_X)_{ij}\circ d.
\end{eqnarray*}
In the second equality, the last term vanishes because
$$
\frac{\tilde f_i^*-\tilde f_j^*}{p}\circ F_{X_i}^*=ob(f)\circ (d F_{X_i}^*)=0.
$$
\end{proof}
Now we turn the above equality into an equality required in the descent condition. The transition function of $V_1$ is given by
$$
a_{ij}:=\exp(ob(F_Y)_{ij}\cdot f_i^*\theta)\cdot\exp (F_{Y_j}^*(ob(f)_{ij}\cdot \theta)),
$$
while the transition function for $V_2$ is given by
$$
b_{ij}:=f_j^*\exp((ob(F_X)_{ij}\cdot \theta)).
$$
Then Lemma \ref{cech relation} implies the commutativity of the following diagram over $Y_{ij}$:
$$
\CD
  V_1|_{Y_i} @>\exp(\nu_{f_i}\cdot \theta)>> V_2|_{Y_i} \\
  @V a_{ij} VV @VVb_{ij}V  \\
   V_1|_{Y_j}  @>>\exp(\nu_{f_j}\cdot \theta)>  V_2|_{Y_j}.
\endCD
$$
The commutativity is nothing but the descent condition for the isomorphism $\exp(\nu_f\cdot\theta)$. So we are done.

\section{Semistability under pullback}
Semistablity is not always preserved under pullback. After all, semistability refers to some given ample line bundle and an ample line bundle does not necessarily pulls back to an ample line bundle. Even worse, in the postive characteristic case, there are well-known examples of semistable vector bundles over curves which pull back to unstable bundles under Frobenius morphism. 

For a polystable Higgs bundle with vanishing Chern classes in characteristic zero,  this is handled by 
the existence of Higgs-Yang-Mills metric-it is a harmonic bundle by this case and harmonic bundles pulls back to harmonic bundles. Consequently, the pullback of a polystable Higgs bundle with vanishing Chern classes is again polystable with vanishing Chern classes. For a semistable Higgs bundle with vanishing Chern classes, one takes a Jordan-H\"older filtration of the Higgs bundle and the semistability of the pullback follows from that of the polystable case.  In the following, we provide a purely algebraic approach to the semistable case. We proceed to the proof of Theorem \ref{semistability}.
\begin{proof}
We resume the notations of Theorem \ref{semistability}. Since taking Chern class commutes with pullback, the statement about vanishing Chern classes of the pullback is trivial. We focus on the semistability below. In the following discussion, we choose and then fix an arbitrary ample line bundle $L$ (resp. $M$) over $X$ (resp. $Y$). We consider first the char $p$ setting. Fix a $W_2(k)$-lifting $\tilde f: \tilde Y_{\log}=(\tilde Y, \tilde B)\to \tilde X_{\log}=(\tilde X, \tilde D)$. First, we observe that the proof of Theorem A.4 \cite{LSZ} works verbatim for a semistable logarithmic Higgs bundle, so that there exists a filtration $Fil_{-1}$ on $E$ such that $\Gr_{Fil_{-1}}(E,\theta)$ is semistable.  Now applying \cite[Theorem A.1]{LSZ}, \cite[Theorem 5.12]{La1} to the \emph{nilpotent} semistable Higgs bundle $\Gr_{Fil_{-1}}(E,\theta)$, we obtain a flow of the following form:
$$
	\xymatrix{
		(E,\theta)\ar[dr]^{Gr_{Fil_{-1}}} && (H_0,\nabla_0)\ar[dr]^{Gr_{Fil_0}} &&  (H_1,\nabla_1)\ar[dr]^{Gr_{Fil_1}} \\
		&(E_0,\theta_0) \ar[ur]^{C_{X_{\log}\subset \tilde X_{\log}}^{-1}} & & (E_1,\theta_1) \ar[ur]^{C_{X_{\log}\subset \tilde X_{\log}}^{-1}}&&
		\cdots,
	}
	$$
in which each Higgs term in bottom is semistable.  Next, because of Corollary \ref{main cor}, we may obtain the pullback flow as follows: 
$$
	\xymatrix{
		f^*(E,\theta)\ar[dr]^{Gr_{f^*Fil_{-1}}} && f^*(H_0,\nabla_0)\ar[dr]^{Gr_{f^*Fil_0}} &&  f^*(H_1,\nabla_1)\ar[dr]^{Gr_{f^*Fil_1}} \\
		&f^*(E_0,\theta_0) \ar[ur]^{C_{Y_{\log}\subset \tilde Y_{\log}}^{-1}} & & f^*(E_1,\theta_1) \ar[ur]^{C_{Y_{\log}\subset \tilde Y_{\log}}^{-1}}&&
		\cdots
	}
	$$
Now as the Higgs terms $(E_i,\theta_i)$s in the first flow are semistable of the same rank and of vanishing Chern classes, the set $\{(E_i,\theta_i)\}_{i\geq 0}$ form a bounded family. So the set $\{f^*(E_i,\theta_i)\}_{i\geq 0}$ also forms a bounded family. In particular, the degrees of subsheaves in $\{f^*E_i\}_{i\geq 0}$ have an upper bound $N$. Suppose $f^*(E,\theta)$ is unstable, that is, there exists a saturated Higgs subsheaf $(F,\eta)$ of positive degree $d$ in $f^*(E,\theta)$. Then $\Gr_{f^*Fil_{-1}}(F,\eta)\subset f^*(E_0,\theta_0)$ is a Higgs subsheaf of degree $d$. It implies that $\Gr_{f^*Fil_0}\circ C^{-1}_{Y_{\log}\subset \tilde Y_{\log}}(F,\eta)\subset f^*(E_1,\theta_1)$ is of degree $pd$.  Iterating this process, one obtains a subsheaf in $f^*(E_i,\theta_i)$ whose degree exceeds $N$. Contradiction. Therefore, $f^*(E,\theta)$ is semistable. 

Now we turn to the char zero case. By the standard spread-out technique, there is a regular scheme $S$ of finite type over $\Z$, and an $S$-morphism $\mathfrak{f}: (\sY,\sB)\to (\sX,\sD)$ and an $S$-relative logarithmic Higgs bundle $(\sE,\Theta)$ over $(\sX,\sD)$, together with a $k$-rational point in $S$ such that $\{\mathfrak f: (\sY,\sB)\to (\sX,\sD), (\sE,\Theta)\}$ pull back to $\{f: (Y,B)\to (X,D), (E,\theta)\}$. In above, we may assume that $\sX$ (resp. $\sY$) is smooth projective over $S$ and $\sD$ (resp. $\sB$) is an $S$-relative normal crossing divisor in $\sX$ (resp. $\sY$).  For a geometrically closed point $s\in S$ and a $W_2(k(s))$-lifting $\tilde s\to S$, we obtain a family $\mathfrak{f}_s: (\sY,\sB)_s\to (\sX,\sD)_s$ over $k(s)$ which is $W_2$-liftable. Once taking an $s\in S$ such that $\textrm{char}(k(s))\geq \rk(\sE_s)=\rk(E)$, we are in the previous char $p$ setting. Hence it follows that $\mathfrak{f}_s^*(\sE,\Theta)_s$ is semistable. From this, it follows immediately that $f^*(E,\theta)$ is also semistable.

\end{proof}

\end{document}